\documentclass[11pt,reqno]{article}
\usepackage{amsfonts}
\usepackage{amsmath}
\usepackage{amssymb}
\usepackage{epsfig}
\usepackage{amsthm}
\usepackage[english]{babel}
\usepackage[hypertex]{hyperref}
\newcommand{\N}{\mathbb N}

\newcommand{\R}{\mathbb R}

\newcommand{\E}{\mathbb E}
\newcommand{\pr}{\mathbb P}

\newcommand{\hier}[3][X]{#1^{(#2)}_{#3}}

\newcommand{\half}{\frac{1}{2}}
\newcommand{\salj}{\mathcal{F}}
\newcommand{\vol}{\mathop{\rm vol}\nolimits}
\newcommand{\Var}{\mathop{\rm Var}\nolimits}
\newenvironment{centre}{\begin{center}}{\end{center}}
\newcommand{\rmd}{\mathrm{d}}

\begin{document}
\theoremstyle{plain}
\newtheorem{thm}{Theorem}
\newtheorem{lem}[thm]{Lemma}
\newtheorem{cor}[thm]{Corollary}
\newtheorem{prop}[thm]{Proposition}

\title{Randomised reproducing graphs}
\author{Jonathan Jordan}

\maketitle

\begin{abstract}We introduce a model for a growing random graph based on simultaneous reproduction of the vertices.  The model can be thought of as a generalisation of the reproducing graphs of Southwell and Cannings and Bonato et al to allow for a random element, and there are three parameters, $\alpha$, $\beta$ and $\gamma$, which are the probabilities of edges appearing between different types of vertices.  We show that as the probabilities associated with the model vary there are a number of phase transitions, in particular concerning the degree sequence.  If $(1+\alpha)(1+\gamma)<1$ then the degree distribution converges to a stationary distribution, which in most cases has an approximately power law tail with an index which depends on $\alpha$ and $\gamma$.  If $(1+\alpha)(1+\gamma)>1$ then the degree of a typical vertex grows to infinity, and the proportion of vertices having any fixed degree $d$ tends to zero.  We also give some results on the number of edges and on the spectral gap.  \end{abstract}

\section{Introduction}

In this paper we introduce a new model for a growing random graph based on simultaneous reproduction of the vertices in the graph, with edges being formed between the new vertices and each other and between the new vertices and the existing ones according to a random mechanism conditioned on the pattern of edges between the vertices in the previously existing graph.  The model is a generalisation of the models introduced by Southwell and Cannings \cite{rc1,rc3,gamenets} and the Iterated Local Transitivity (ILT) model of \cite{ILT}, introducing stochasticity, which causes the regular structure found in the graphs of \cite{rc1,rc3,gamenets} to be lost, and which may make them more suitable for modelling in areas such as social networks; the authors of \cite{ILT} particularly suggest their model as a model for online social networks, mentioning Facebook and Twitter among other examples.

We will show that our model, which depends on three parameters $\alpha,\beta$ and $\gamma$, which are to be thought of as probabilities, exhibits a number of phase transitions as the parameters vary; for example for some values of the parameters we will show that the degree distribution of the graph converges to a limiting probability distribution, while for other choices of the parameters the degree of a randomly chosen (in an appropriate sense) vertex in $G_n$ can be shown to tend to infinity as $n\to\infty$.  We will also show that for certain choices of the parameter values the model exhibits a power-law-like decay of the degree distribution, which is a property reported for many ``real world'' networks, and is also associated with other random graph models such as preferential attachment.

We start with a graph $G_0$, and form a new graph $G_{n+1}$ by adding a ``child'' vertex for every vertex of
$G_n$.  As in \cite{rc1,rc3,gamenets} we denote the vertices by binary strings, writing $v0$ for the ``child'' of vertex $v\in V(G_n)$ and $v1$ for the continuation of vertex $v$ as a vertex of $G_{n+1}$.  The edges of $G_{n+1}$ are then obtained according to the following mechanism.
For each $n$ define independent (of each other and of the random variables at other stages of the construction) Bernoulli random variables $\hier[a]{n}{\{u,v\}}\sim Ber(\alpha)$ for each unordered pair $\{u,v\}$ of vertices of $G_n$, $\hier[b]{n}{u}\sim Ber(\beta)$ for each vertex in $G_n$, $\hier[c]{n}{(u,v)}\sim Ber(\gamma)$ for each ordered pair $(u,v)$ of vertices of $G_n$, and connect vertices as follows:\begin{description}\item[(a)] $u1$ is connected to $v1$ in $G_{n+1}$ if and only if $u$ and $v$ are connected in $G_n$, that is existing edges are retained, and no further edges are formed between existing vertices.
\item[(b)] $u0$ is connected to $u1$ in $G_{n+1}$ if and only if $\hier[b]{n}{u}=1$, so each child is connected to its parent with probability $\beta$.
\item[(c)] $u0$ is connected to $v1$ in $G_{n+1}$ if and only if $\hier[c]{n}{(u,v)}=1$ and $u$ and $v$ are connected in $G_n$, so each child is connected to each of its parent's neighbours with
probability $\gamma$.
\item[(d)] $u0$ is connected to $v0$ in $G_{n+1}$ if and only of $\hier[a]{n}{\{u,v\}}=1$ and $u$ and $v$ are connected in $G_n$, so each child is connected to each of its
parent's neighbours' children with probability $\alpha$.
\end{description}

The models introduced in \cite{rc1,rc3,gamenets} have $\alpha,\beta,\gamma\in\{0,1\}$, so are deterministic.  Additionally the case where $\alpha=0,\beta=1,\gamma=1$ is the ILT model, introduced in \cite{ILT} as a model for online social networks.  The ILT$(p)$ model introduced in \cite{ILT} as a stochastic generalisation of the ILT model adds extra random edges between the child vertices without regard to whether their parents were connected, and thus cannot be seen as a special case of our model.  In addition as defined in \cite{ILT} the ILT$(p)$ model always has at least the edges found in the basic ILT model, so is not suited to producing relatively sparse graphs.

The model differs from duplication graphs, for example those considered in \cite{biodupe,jorddupe}, in that in those models only one vertex, chosen at random, duplicates at any one time step, whereas in the models considered here and in \cite{ILT,rc1,rc3,gamenets} all vertices simultaneously duplicate.

Our main results concern the degree distribution.  We will deal with the cases where $\beta=0$ and $\beta>0$ separately, as the behaviour of the model when $\beta=0$ is potentially quite different, with large numbers of isolated vertices.

\begin{thm}\label{beta0}
Let $\beta=0$.  Then, if $(1+\gamma)(\alpha+\gamma)\leq 1$, the probability that a randomly chosen vertex in the graph $G_n$ is isolated tends to $1$ as $n\to\infty$, and the proportion of vertices in the graph with degree zero tends to $1$, almost surely.  If $(1+\gamma)(\alpha+\gamma)>1$, then the probability that a randomly chosen vertex in the graph $G_n$ is isolated converges to some value strictly less than $1$.
\end{thm}

\begin{thm}\label{betanot0}
Assume $\beta>0$, and let $\hier[p]{n}{d}$ be the proportion of vertices in $G_n$ with degree $d$.  Then \begin{description}\item[(a)] If $(1+\gamma)(\alpha+\gamma)< 1$ there exists a random variable $X$ such that $\hier[p]{n}{d}\to P(X=d)$ as $n\to\infty$, almost surely.
\item[(b)] Under the conditions of (a), the random variable $X$ has a finite $p$th moment if $(1+\gamma)^p+(\alpha+\gamma)^p<2$, and does not have a finite $p$th moment if $(1+\gamma)^p+(\alpha+\gamma)^p>2$.
\item[(c)] If $(1+\gamma)(\alpha+\gamma)> 1$ then $\hier[p]{n}{d}\to 0$ as $n\to\infty$, almost surely.
\end{description}
\end{thm}

Note that Theorem \ref{betanot0}(b) implies that if $(1+\gamma)^p+(\alpha+\gamma)^p=2$ the tail of the degree distribution is asymptotically close to a power law degree distribution with index given by $-(p+1)$, in the sense that $q$th moments exist for $q<p$ but not for $q>p$.

In \cite{ILT}, it is shown that the ILT model exhibits a ``densification power law'', which is defined to mean that, if $E_n$ is the number of edges of $G_n$ and $V_n$ the number of vertices, then $E_n$ is proportional to $(V_n)^a$ for some $a\in(1,2)$.  The following result shows that our model exhibits a phase transition in this respect, with the transition occurring where $2\gamma+\alpha=1$.  Note that in our model, as in the ILT model, $V_n=2^nV_0$ for all $n$.

\begin{thm}\label{dense}\begin{description}\item[(a)] If $2\gamma+\alpha>1$ then $W_n=\frac{E_n}{(1+2\gamma+\alpha)^n}$ converges to a positive limit, so that the model has a densification power law as defined by \cite{ILT} with exponent $\frac{\log(1+2\gamma+\alpha)}{\log 2}$.
\item[(b)] If $2\gamma+\alpha<1$ then $$\frac{E_n}{2^n}\to \frac{V_0\beta}{1-2\gamma-\alpha},$$ almost surely, as $n\to\infty$, so that the number of edges grows at the same rate as the number of vertices
\item[(c)] If $2\gamma+\alpha=1$ then $$\frac{E_n}{2^n n}\to \frac{V_0\beta}{2},$$ almost surely, as $n\to\infty$.\end{description}
\end{thm}

Note that the combination of Theorems \ref{dense} and \ref{betanot0} implies that when $2\gamma+\alpha>1$ but $(1+\gamma)(\alpha+\gamma)<1$ the process exhibits both a densification power law in the sense of \cite{ILT} and an approximately power law limit for the degree distribution.

A further result in \cite{ILT} on the ILT model concerns the spectral gap.  They show that the normalised graph Laplacian $\mathcal{L}$, as defined by Chung \cite{chung}, of the ILT model has a large spectral radius, defined as $\max\{|\lambda_1-1|,|\lambda_{n-1}-1|\}$, where $\lambda_1$ is the second smallest eigenvalue (the smallest being $\lambda_0=0$ for any graph) and $\lambda_{n-1}$ is the largest eigenvalue and thus that the graph has relatively poor expansion properties.  The following results show that the same is also true for our model.  We concentrate on the case $\beta=1$, where the graphs are connected; otherwise $\lambda_1$ will be zero.  The proofs use the Cheeger constant and its relationship to $\lambda_1$, as defined in Chapter 2 of Chung \cite{chung}.

\begin{thm}\label{gap}Let $\beta=1$ and assume that $G_0$ is connected, so that $G_n$ will also be connected for all $n$.  Let $\lambda_1(G_n)$ be the smallest non-negative eigenvalue of the Laplacian of $G_n$.  Then \begin{description}\item[(a)] If $2\gamma+\alpha\leq1$ then $\lambda_1(G_n)\to 0$ as $n\to\infty$.
\item[(b)] If $2\gamma+\alpha> 1$ then there exists a (random) $\Lambda$, with $\Lambda<1$ almost surely, such that $\limsup_{n\to\infty} \lambda_1=\Lambda$.\end{description}
\end{thm}

\section{Pictures}

This section shows a few examples of graphs of this type, which were generated using a script written with the igraph package, \cite{igraph}, in R.  The first two pictures show examples with $n=7$, $\alpha=0$, $\beta=1$ and $\gamma=0.2$ and
$0.49$ respectively.

\begin{centre}
\begin{tabular}{cc}\scalebox{0.5}{{\includegraphics{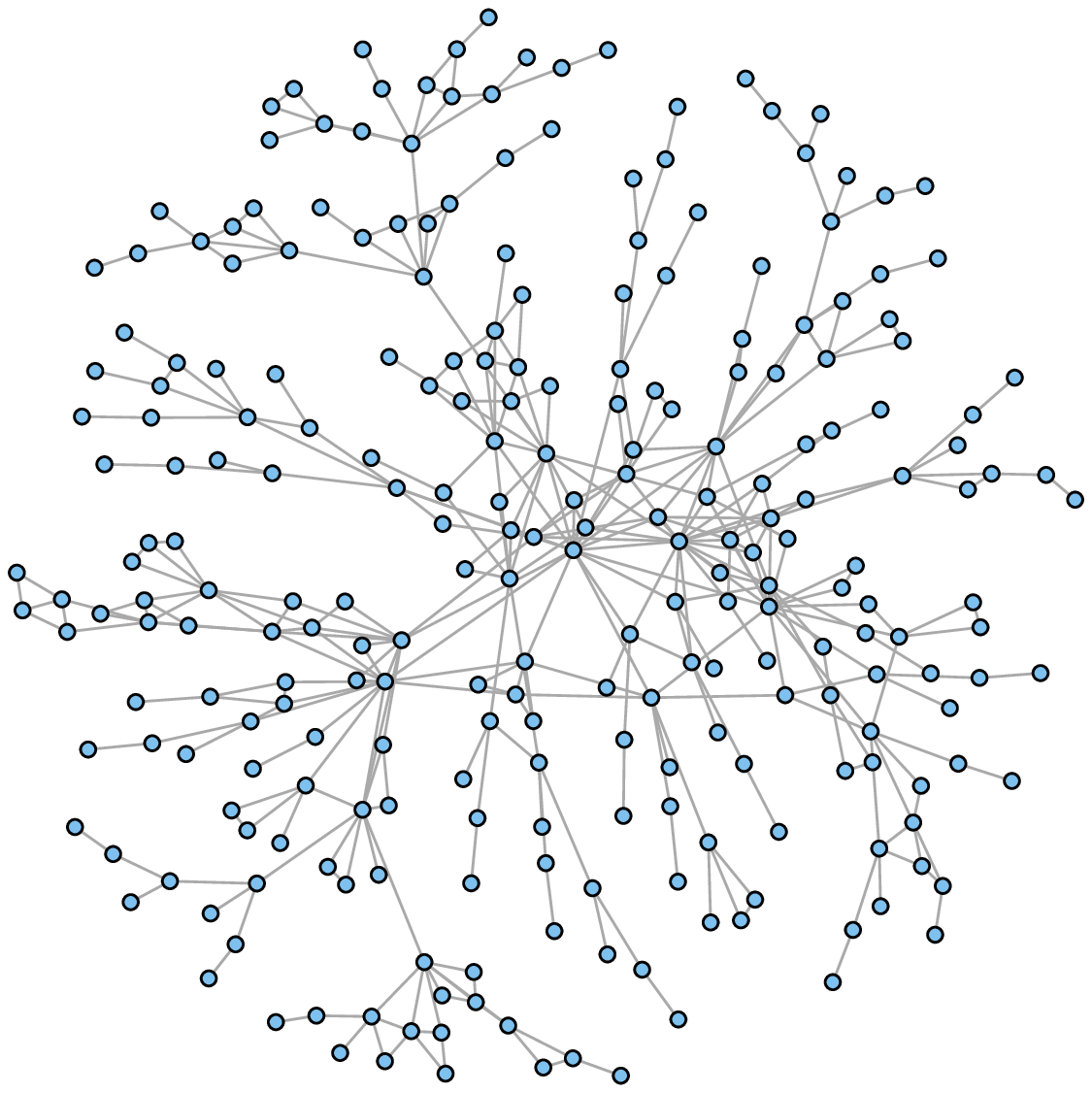}}} &
\scalebox{0.5}{{\includegraphics{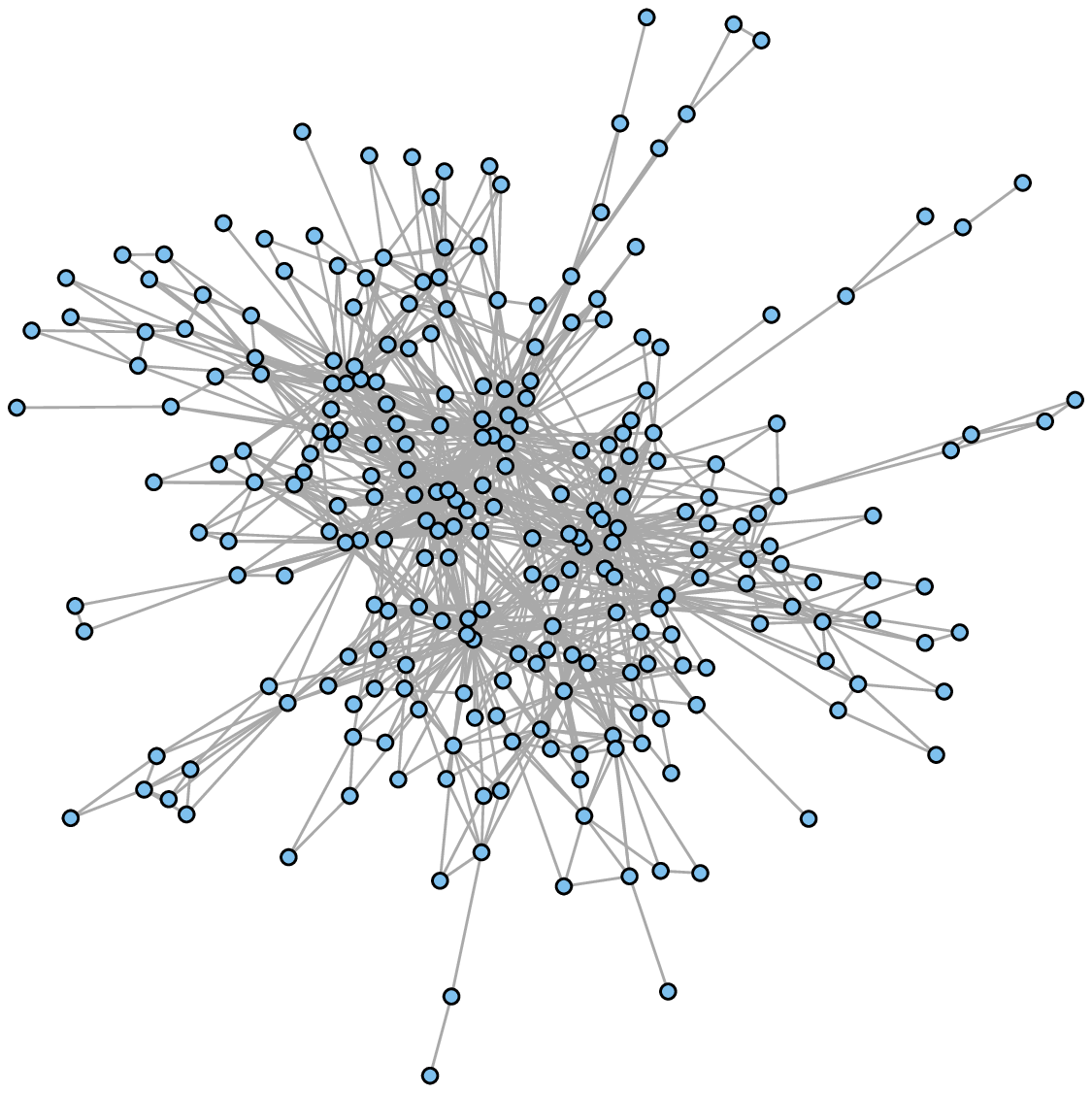}}} \end{tabular}
\end{centre}

The next two pictures have $\alpha=0, \gamma=0.366 (\approx
\frac{\sqrt{3}-1}{2})$, and again $n=7$.  The first has $\beta=0.5$, the second
$\beta=0.8$.
\begin{centre}
\begin{tabular}{cc}\scalebox{0.5}{{\includegraphics{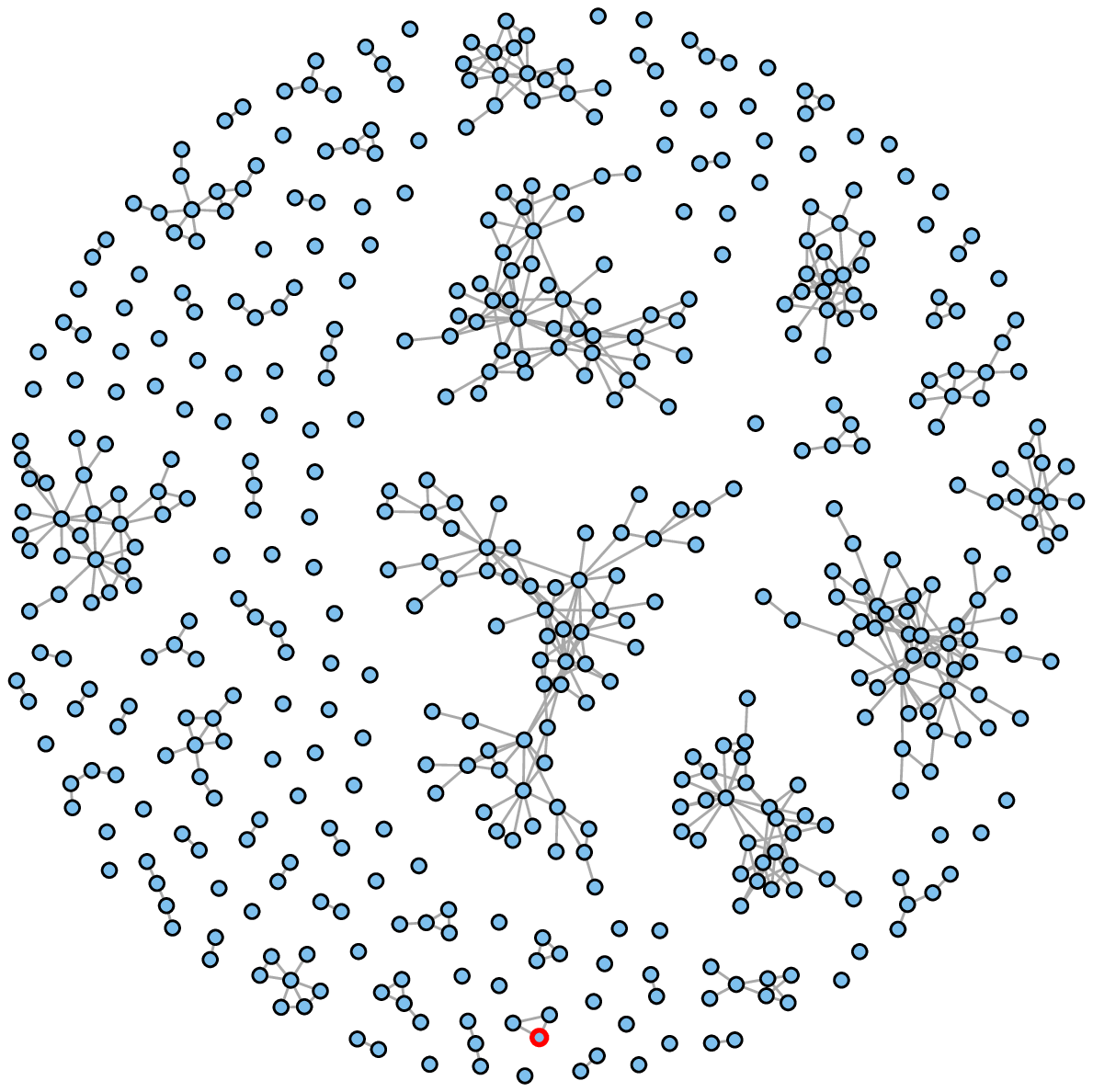}}} &
\scalebox{0.5}{{\includegraphics{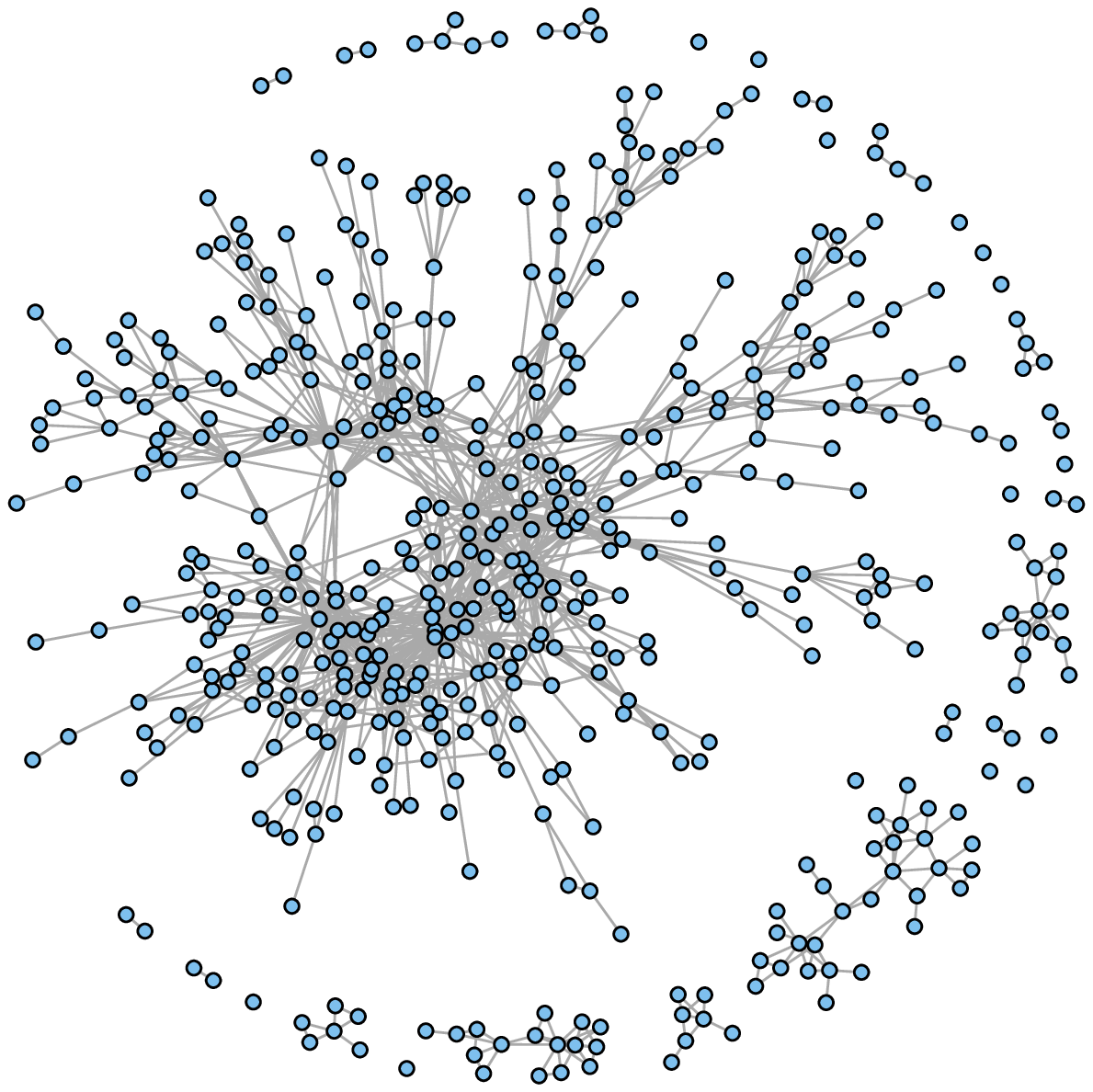}}} \end{tabular}
\end{centre}

\section{Proofs of Theorems}

We start with a lemma on the conditional expectation and variance of the number of edges in $G_n$, $E_n$.  This lemma will be useful for obtaining the mean of the stationary distribution of a Markov chain which we will use to prove Theorems \ref{beta0} and \ref{betanot0}.  We define $\salj_n$ to be the $\sigma$-algebra generated by the graphs $G_m$ for $m\leq n$.

\begin{lem}\label{subgraphs}
The conditional expectation and variance of $E_{n+1}$ satisfy
\begin{eqnarray*} \E(E_{n+1}|\salj_n) &=& (1+2\gamma+\alpha)E_n + 2^n\beta V_0 \\ \Var(E_{n+1}|\salj_n) &=& E_n(2\gamma(1-\gamma)+\alpha(1-\alpha))+2^nV_0\beta(1-\beta)\end{eqnarray*}
\end{lem}

\begin{proof}
This follows from the fact that the $E_{n+1}$ can be written $E_n+E_{n+1,1}+E_{n+1,2}+E_{n+1,3}$ where $E_{n+1,1}$, $E_{n+1,2}$ and $E_{n+1,3}$ are independent, $E_{n+1,1}$ represents the edges between parents and children of their neighbours and, conditional on $\salj_n$ has a $Bi(2E_n,\gamma)$ distribution, $E_{n+1,2}$ represents the edges between children of neighbouring vertices and, conditional on $\salj_n$ has a $Bi(E_n,\alpha)$ distribution, and $E_{n+1,3}$ represents the edges between parents and their children and, conditional on $\salj_n$ has a $Bi(V_n,\beta)$ distribution. As $V_n=2^n V_0$ the result follows.
\end{proof}

\subsection{Proof of Theorem \ref{dense}}

We start with (a).  The following approach is based on that in \cite{athreya} for multitype branching processes, the idea being that the vertices and edges in the graph $G_n$ can be thought of as the two types in a population undergoing branching.  However the resulting multitype branching process is not irreducible, so the results in \cite{athreya} cannot be used directly.

Given an edge in $G_m$ between vertices $u$ and $v$, there will be an edge in $G_{m+1}$ between $u1$ and $v_1$, and in addition there will be edges between $u1$ and $v0$ and $v1$ and $u0$ each with probability $\gamma$ and an edge between $u0$ and $v0$ with probability $\alpha$.  We can consider these edges as offspring of the edge between $u$ and $v$, and thus consider the set of edges in $G_n$ (for $n>m$) which are descendants of the edge between $u$ and $v$ in $G_m$ as a generation in a Galton-Watson branching process with offspring mean $1+2\gamma+\alpha$, and where the extinction probability is zero and the number of offspring bounded.  Treating the descendants of a given edge in $G_m$ as a subset of the edge set of $G_n$, this shows that $\liminf \frac{E_n}{(1+2\gamma+\alpha)^n}$ is a positive random variable.

Now define $$W_n=\frac{V_n+\frac{2\gamma+\alpha-1}{\beta}E_n}{(1+2\gamma+\alpha)^n}.$$  Then $\E(W_{n+1}|\salj_n)$, so $(W_n)_{n\in \N}$ is a non-negative martingale, and thus almost surely has a non-negative limit $W$.  The above conclusion on $\liminf \frac{E_n}{(1+2\gamma+\alpha)^n}$ shows that $P(W=0)=0$, giving the result.

For (b), Lemma \ref{subgraphs} shows that $\E(E_n)=\frac{V_0\beta}{1-2\gamma-\alpha}2^n+o(2^n)$ and $$\Var(E_{n})=2^{n-1}\frac{V_0\beta}{1-2\gamma-\alpha}(2\gamma(1-\gamma)+o(2^n)+\alpha(1-\alpha))+2^nV_0\beta(1-\beta)+(1+2\gamma+\alpha)^2
\Var(E_{n-1}),$$ which shows that $$\Var\left(\frac{E_n}{2^n}\right)=O\left(\left(\max\left(\half,\frac{(1+2\gamma+\alpha)^2}{4}\right)\right)^n\right),$$ which implies the result via Chebyshev's inequality and the Borel-Cantelli Lemmas.

For (c), an iterative use of Lemma \ref{subgraphs} shows that $\E(E_n)=2^n\left(E_0+\frac{\beta V_0}{2}n\right)$ and $\Var(E_{n+1})=2^{n-1}n\beta V_0 (2\gamma(1-\gamma)+\alpha(1-\alpha)+4)+O(2^n)$.  Then $$\Var\left(\frac{E_n}{2^n n}\right)=O\left(\frac{1}{n^2 2^n}\right),$$ allowing the Chebyshev/Borel-Cantelli argument again.

\subsection{Proof of Theorem \ref{gap}}
We consider some small $m$, and find the Cheeger constant of $G_m$.  By the definition in \cite{chung}, this will be $e(S_m,\bar{S_m})/\vol(S_m)$ for some $S_m\subseteq V(G_m)$, where for two subsets of the vertex set $S$ and $S'$ $e(S,S')$ is the number of edges between a vertex in $S$ and one in $S'$, and $\vol(S)$ is the sum of the degrees of vertices in $S$.  (Note that $\vol(S)=2e(S,S)+e(S,\bar{S}$.)  Now consider the descendants of $S_m$ in $G_n$ as a subset $S_n\subseteq V(G_n)$.  Then the same arguments as in the proof of Theorem \ref{dense}, applied to the subgraphs descending from $S_m$ and $\bar{S_m}$, show that if $2\gamma+\alpha<1$ then the $e(S_n,S_n)$ and $e(\bar{S_n},\bar{S_n})$ both grow at rate $2^n$, in the sense that $\frac{e(S_n,S_n)}{2^n}$ and $\frac{e(\bar{S_n},\bar{S_n})}{2^n}$ converge almost surely to positive constants as $n\to\infty$, and similarly if $2\gamma+\alpha=1$ $e(S_n,S_n)$ and $e(\bar{S_n},\bar{S_n})$ both grow at rate $2^n n$, and if $2\gamma+\alpha>1$ $e(S_n,S_n)$ and $e(\bar{S_n},\bar{S_n})$ both grow at rate $(1+2\gamma+\alpha)^n$.

Next, again as in the proof of Theorem \ref{dense}, $(e(S_n,\bar{S_n}))_{n\in\N}$ forms a Galton-Watson branching process with mean of the offspring distribution $1+2\gamma+\alpha$, and extinction probability zero, so $e(S_n,\bar{S_n})$ will grow at rate $(1+2\gamma+\alpha)^n$.  So for $2\gamma+\alpha> 1$ $\frac{e(S_n,\bar{S_n})}{\min(\vol(S_n),\vol(\bar{S_n}))}$, which by the definition in \cite{chung} is greater than the Cheeger constant of $G_n$, converges to a constant (less than $1$, as $\vol(S_n)=2e(S_n,S_n)+e(S_n,\bar{S_n})$) and this constant bounds the lim sup of the Cheeger constant of $G_n$ above.

In the case where $2\gamma+\alpha<1$ $\frac{e(S_n,\bar{S_n})}{\max(\vol(S_n),\vol(\bar{S_n}))}=O\left(\left(\frac{1+2\gamma+\alpha}{2}\right)^n\right)\to 0$ as $n\to\infty$, and hence so is the Cheeger constant.  Similarly if $2\gamma+\alpha=1$ $\frac{e(S_n,\bar{S_n})}{\max(\vol(S_n),\vol(\bar{S_n}))}=O\left(\frac{1}{n}\right)$, and again so is the Cheeger constant.

Hence by the Cheeger inequality (Lemma 2.1 and Theorem 2.2 of \cite{chung}), $\limsup_{n\to\infty}\lambda_1(G_n)<1$ almost surely in the case $2\gamma+\alpha> 1$, and $\lambda_1(G_n)$ tends to zero in the case $2\gamma+\alpha\leq 1$.

\subsection{Proofs of Theorems \ref{beta0} and \ref{betanot0}}

The proofs of Theorems \ref{beta0} and \ref{betanot0} will rely on defining a certain Markov chain whose value $X_n$ represents the degree of a random vertex in the graph $G_n$.  We construct this by letting $v_0$ be a vertex of $G_0$ chosen uniformly at random, and then, using the binary string notation for the vertices described above, for $n\geq 1$ let $v_{n}=v_{n-1}1$ with probability $1/2$ and
letting $v_n=v_{n-1}0$ with probability $1/2$.  We then let $X_n$ be the degree of $v_n$ in $G_n$.

Then
\begin{equation}\label{xrec}X_{n+1}=\xi_{n+1}X_n+(1-\xi_{n+1})W_{n+1}+
Y_{n+1}+Z_{n+1},\end{equation} where, conditional on $G_n$,
$Y_{n+1}\sim Bin(X_n,\gamma)$, $W_{n+1}\sim Bin(X_n,\alpha)$,
$Z_{n+1}\sim Bin(1,\beta)$ and $\xi_{n+1}\sim Bin(1,\half)$, with
all these variables being conditionally independent given $G_n$.

Here, $\xi_{n+1}=1$ if our vertex in $G_{n+1}$ is a parent and $0$
if it is a child, $W_{n+1}$ represents child-child connections (so
does not appear if $\xi_{n+1}=1$), $Y_{n+1}$ represents
connections between a child and its parents' neighbours, and
$Z_{n+1}$ represents the connection between the child and its parent.

As defined above, $(X_n)_{n\in \N}$ is a discrete time Markov
chain on the natural numbers (including zero if $\beta<1$). It is
irreducible and aperiodic if $\beta>0$, $\alpha<1$ and $\gamma<1$. (If $\beta=0$
then zero is an absorbing state, and if either $\alpha$ or
$\gamma$ is $1$ then $X_n$ is increasing in $n$ and so the chain
is certainly not irreducible, but otherwise $P(X_{n+1}=1|\salj_n)$ is always positive.)

\begin{prop}If $2\gamma+\alpha<1$ the distribution of $X_n$ converges in the Wasserstein-$1$ metric to a unique fixed point with finite mean $\frac{2\beta}{1-2\gamma-\alpha}$.\end{prop}

\begin{proof}
Note that if we have another random variable $\hat{X}_n$ with a different distribution on $\N_0$, we can apply \eqref{xrec} to it by defining, conditional on $\hat{X}_n$, $\hat{Y}_{n+1}\sim Bin(\hat{X}_n,\gamma)$ and $\hat{W}_{n+1}\sim Bin(\hat{X}_n,\alpha)$ using the same set of Bernoulli trials as for $Y_{n+1}$ and $W_{n+1}$ respectively, and letting $$\hat{X}_{n+1}=\xi_{n+1}\hat{X}_n+(1-\xi_{n+1})\hat{W}_{n+1}+
\hat{Y}_{n+1}+Z_{n+1},$$ in which case
$\E(|X_{n+1}-\hat{X}_{n+1}||\salj_n)=\half(2\gamma+\alpha+1)|X_{n}-\hat{X}_{n}|$,
so that we have a contraction in the Wasserstein metric if
$2\gamma+\alpha<1$. Hence in this case there is convergence in the
Wasserstein-$1$ metric of the degree
distributions to a unique fixed point with finite mean.

We can calculate
the mean of this distribution by using Lemma \ref{subgraphs}: letting $m=2$ we get $$\E(\hier[k]{n+1}2)=(2\gamma+\alpha+1)\hier[k]n2+2^n\beta v_0,$$ where $v_0$ is the number of vertices in the initial graph, and solving this we find that the expected
number of edges in $G_n$ is
$$\frac{\beta v_0(2^n-(1+2\gamma+\alpha)^n)}{1-2\gamma-\alpha},$$ so (as the number of vertices in $G_n$ is $2^n v_0$) the expected average
degree is
$$\frac{\beta(2^n-(1+2\gamma+\alpha)^n)}{2^{n-1}(1-2\gamma-\alpha)},$$ which converges to
$\frac{2\beta}{1-2\gamma-\alpha}$ as $n\to\infty$.
\end{proof}

To go further than this we use Foster-Lyapunov techniques, as described in Meyn and Tweedie \cite{meyntweedie} in the more general case of an uncountable state space.  The following lemma on the conditional moments of $X_{n+1}$ (including negative and fractional moments) will be useful.

\begin{lem}\label{moments}Let $p\in\R$.  Then as $x\to\infty$, $$\E\left(\left(\frac{1+x+Y_{n+1}+Z_{n+1}}{1+x}\right)^p|X_n=x\right)\to (1+\gamma)^p,$$ and
$$\E\left(\left(\frac{1+W_{n+1}+Y_{n+1}+Z_{n+1}}{1+x}\right)^p|X_n=x\right)\to (\alpha+\gamma)^p.$$
\end{lem}

\begin{proof}In the case where $p<0$ this is a special case of Theorem 2.1 of Garc\'{i}a and Palacios in \cite{negativemoments}.  When $p>0$ we can adapt the argument of Theorem 2.1 in \cite{negativemoments} with the additional condition that the random variables are bounded above by a constant multiple of their mean, which is satisfied in this case.  We assume that $(A_n)_{n\in\N}$ is a sequence of positive random variables with associated sequences of constants $\mu_n$ and $\sigma_n$ such that $A_n\leq k\mu_n$ with probability $1$ for some $k$ and for all $n$, and that the conditions of Theorem 2.1 in \cite{negativemoments} hold.

For the upper bound, \begin{eqnarray*}\E(Y_n^p) &=& \E(Y_n^pI_{\{Y_n<\mu_n+\mu_n^\delta\}})+\E(Y_n^pI_{\{Y_n\geq \mu_n+\mu_n^\delta\}}) \\ &\leq & (\mu_n+\mu_n^\delta)^p+(k\mu_n)^p\pr(Y_n\geq \mu_n+\mu_n^\delta,\end{eqnarray*} with order of magnitude $\mu_n^p$.

For the lower bound (which is obvious by Jensen's inequality if $p\geq 1$), $$\E(Y_n^p)\geq (\mu_n-\mu_n^{\delta})^p\pr(X_n>\mu_n-\mu_n^\delta),$$ again with order of magnitude $\mu_n^p$.
\end{proof}

\begin{prop}\label{posrec}If $(1+\gamma)(\alpha+\gamma)<1$, the Markov chain is positive recurrent, and thus the distribution of $X_n$ converges to a stationary distribution.
\end{prop}

\begin{proof}
This uses Theorem 11.0.1 of \cite{meyntweedie}.

We choose $p\in(0,1)$ such that $(1+\gamma)^p+(\alpha+\gamma)^p<2$.  Because $\frac{\rmd}{\rmd p}((1+\gamma)^p+(\alpha+\gamma)^p)$ is negative at $p=0$ if $\log(1+\gamma)+\log(\alpha+\gamma)<0$, it will be possible to find such a $p$ if $(1+\gamma)(\alpha+\gamma)<1$.

We now let $V(x)=x^p$.  In \cite{meyntweedie}, the drift $\Delta V(x)$ is defined as $$\Delta V(x)=\E(V(X_{n+1})-V(X_n)|X_n=x),$$ and by Theorem 11.0.1 of \cite{meyntweedie} the chain will be positive recurrent if (for some $V$) $\Delta V(x)\leq -1$ for $x$ large enough.  Now \begin{eqnarray*}\E(X_{n+1}^p|X_n=x) &=& \frac{x^p}{2}\left(\E\left(\left(1+
\frac{Y_{n+1}}{x}+\frac{Z_{n+1}}{x}
\right)^p|G_n\right)+\E\left(\left(\frac{W_{n+1}}{x}+
\frac{Y_{n+1}}{x}+\frac{Z_{n+1}}{x}
\right)^p|G_n\right)\right)\\ & \leq & \frac{x^p}{2}\left(\left(1+\gamma\right)^p+ \left(\alpha+\gamma\right)^p\right)+o(x^p) \\ & & \mbox{(by Lemma \ref{moments})},\end{eqnarray*} so $$\Delta V(x)\leq x^p\left(\frac{\left(1+\gamma\right)^p+ \left(\alpha+\gamma\right)^p}{2}-1\right)+o(x^p),$$ which will be less than $-1$ for $x$ large enough, giving the result.

\end{proof}

We now investigate the tail behaviour of the stationary distribution, in the case where Proposition \ref{posrec} shows one exists.

\begin{prop}\label{tail1}
Let $p>0$.  If $(1+\gamma)^p+(\alpha+\gamma)^p<2$, then a random variable $X$ with the stationary distribution of the chain has finite $p$th moment $\E(X^p)$, and we have convergence of $p$th moments, $\E(X_n^p)\to \E(X^p)$ as $n\to\infty$.
\end{prop}

\begin{proof}
Again this uses a Foster-Lyapunov type technique, in this case Theorem 14.0.1 of \cite{meyntweedie} which states that if, for a given function $f\geq 1$, we can find $V$ such that $\Delta V(x)<-f(x)$ for $x$ large enough then $f$ has a finite integral with respect to the stationary distribution and that $\E(f(X_n))$ converges to this integral.  We will set $f(x)=x^p+1$.

Let $V(x)=kx^p$, where $k$ is chosen so that $$k\left(\frac{\left(1+\gamma\right)^p+ \left(\alpha+\gamma\right)^p}{2}-1\right)<-1.$$  Then, by Lemma \ref{moments}, $$\Delta V(x)\leq kx^p\left(\frac{\left(1+\gamma\right)^p+\left(\alpha+\gamma\right)^p}{2}-1\right)+o(x^p),$$ and so $\Delta V(x)\leq -f(x)$ for $x$ large enough, giving the result.
\end{proof}

\begin{prop}\label{tail2}
Let $p>0$.  If $(1+\gamma)^p+(\alpha+\gamma)^p>2$, then a random variable $X$ with the stationary distribution of the chain does not have finite $p$th moment $\E(X^p)$.
\end{prop}

\begin{proof}
As $Z_{n+1}\geq 0$, we have $$\E(X_{n+1}^p|X_n=x)\geq \frac{x^p}{2}\left(\E\left(\left(1+\frac{Y_{n+1}}{x}\right)^p|X_n=x\right)+\E\left(\left( \frac{W_{n+1}}{x}+\frac{Y_{n+1}}{x}\right)^p|X_n=x\right)\right),$$ so by Lemma \ref{moments} $$\E(X_{n+1}^p|X_n=x)\geq \frac{x^p}{2}\left((1+\gamma)^p+(\alpha+\gamma)^p\right)+o(x^p).$$  Hence the $p$th moment of $X_n$ tends to infinity as $n\to\infty$, so by Theorem 14.0.1 of \cite{meyntweedie}, again applied to $f(x)=x^p+1$, the stationary distribution cannot have a finite $p$th moment.
\end{proof}

\begin{prop}If $\beta>0$ and $(1+\gamma)(\alpha+\gamma)>1$ the Markov chain is transient.\end{prop}

\begin{proof}
By Lemma \ref{moments}, \begin{equation}\label{negmom}\frac{E\left(\left(1+X_{n+1}\right)^p|X_n=x\right)}{(1+x)^p}\to \frac{(1+\gamma)^p+(\alpha+\gamma)^p}{2},\end{equation} so we apply Theorem 8.0.2 (i) of \cite{meyntweedie} with $V(x)=1-(1+x)^p$ for some $p<0$ such that $(1+\gamma)^p+(\alpha+\gamma)^p<2$.  With this choice of $V$, \eqref{negmom} shows that $\Delta V(x)>0$ for $x$ large enough, and as $V$ is bounded and positive on the natural numbers Theorem 8.0.2 (i) of \cite{meyntweedie} gives the result.

\end{proof}

The case where $\beta=0$ is something of a special case as the chain is not irreducible.  However we can show that when $(1+\gamma)(\alpha+\gamma)\leq 1$ the probability that a randomly chosen vertex is isolated tends to $1$, while there is positive probability that a randomly chosen vertex is not isolated when $(1+\gamma)(\alpha+\gamma)>1$.

\begin{prop}\label{beta0p}
If $\beta=0$, then \begin{enumerate} \item if $(1+\gamma)(\alpha+\gamma)\leq 1$ then almost surely $X_n=0$ for $n$ sufficiently large, and the proportion of isolated vertices in $G_n$ tends to $1$ almost surely as $n\to\infty$;\item if $(1+\gamma)(\alpha+\gamma)> 1$ then there is $q>0$ such that the probability that $X_n\to \infty$  as $n\to\infty$ is $q$ and the probability that $X_n\to 0$ as $n\to \infty$ is $1-q$.\end{enumerate}
\end{prop}

\begin{proof}
We note that $(X_n)$ follows a Smith-Wilkinson branching process in random environment, \cite{swbpre}.  The environmental variables which determine the random environment are the random variables $\xi_n$, with the offspring distribution of the branching process at time $n$ having mean $1+\gamma$ if $\xi_{n+1}=1$ and $\alpha+\gamma$ if $\xi_{n+1}=0$.  Hence, by Theorem 3.1 of \cite{swbpre}, the branching process dies out with probability $1$ if $\half\log(1+\gamma)+\half\log(\alpha+\gamma)\leq 0$, i.e. if $(1+\gamma)(\alpha+\gamma)\leq 1$, and the branching process dies out with probability strictly less than $1$ otherwise, hence there is positive probability that $X_n\to \infty$  as $n\to\infty$.  To see that the proportion of isolated vertices tends to $1$ almost surely when $(1+\gamma)(\alpha+\gamma)\leq 1$, note that the proportion of isolated vertices is increasing (as if a vertex $v$ is isolated in $G_n$ both $v0$ and $v1$ are isolated in $G_{n+1}$) and therefore must converge to some value, which cannot be less than $1$ as the degree of a random vertex converges to zero almost surely.
\end{proof}

\begin{prop}\label{as}
\begin{description}
\item[(a)]
If $\beta>0$ and $(1+\gamma)(\alpha+\gamma)<1$, the degree distribution of the graph converges to the stationary distribution of the Markov chain in the sense that if we let $\hier[p]{n}{d}$ be the proportion of vertices in $G_n$ with degree $d$, and let $X$ be a random variable with the stationary distribution of the Markov chain, then $\hier[p]{n}{d}\to P(X=d)$ as $n\to\infty$, almost surely, for all $d\in\N_0$.
\item[(b)] If $\beta>0$ and $(1+\gamma)(\alpha+\gamma)>1$ then $\hier[p]{n}{d}\to 0$ as $n\to\infty$, almost surely, for all $d\in\N_0$.
\end{description}
\end{prop}

\begin{proof}
The graph at stage $r$ contains $2^rv_0$ vertices.  We then consider the edges of $G_{r+s}$ in two sets: those which are between descendants of the same vertex in $G_r$, and those which are between descendants of different vertices in $G_r$.  For the former, the appearance of edges between descendants of one given vertex is independent of what happens to the descendants of the other vertices, so we can model these edges of $G_{r+s}$ as consisting of $2^rv_0$ independent copies of $\tilde{G}_s$, where $(\tilde{G}_n)_{n\in\N}$ represents the process as evolved from a single vertex with no edges, $\tilde{G}_0$.  Then, by Chebyshev's inequality and a Borel-Cantelli argument, as $r\to\infty$ proportions of vertices in $G_{r+s}$ with degree $d$ excluding connections to descendants of different vertices in $G_r$ converge, almost surely, to $P(X_s=d)$.  In the $(1+\gamma)(\alpha+\gamma)>1$ case we know that $P(X_s=d)\to 0$ as $s\to\infty$ for all $d$, and the actual degree of a vertex is bounded below by the degree excluding some connections, so this is enough to prove (b).

To complete the proof of (a), we need to consider edges between vertices which are descendants of different vertices in $G_r$.  We couple the process, starting from $G_r$, with a process with $\beta=0$ by removing all edges between a vertex and its offspring, and all edges descended from such edges.  The edges thus removed from $G_{r+s}$ will all be between vertices descended from the same vertex in $G_r$, so all edges between vertices descended from different vertices in $G_r$ are present in the $\beta=0$ version.  But by Proposition \ref{beta0p} the proportion of vertices in $G_{r+s}$ which have non-zero degree in the $\beta=0$ version tends to zero as $s\to\infty$, and so this also applies to the proportion of vertices in $G_{r+s}$ which have edges connecting them to vertices with a different ancestor in $G_r$.

Hence as both $r$ and $s$ $\to\infty$ the proportion which have degree $d$ converges to $P(X_s=d)$.
\end{proof}

Finally we can put the Propositions above together to deduce Theorems \ref{beta0} and \ref{betanot0}.

\textit{Proof of Theorem \ref{beta0}.} Theorem \ref{beta0} follows from Proposition \ref{beta0p}; in the supercritical case where $(1+\gamma)(\alpha+\gamma)>1$ the probability that a randomly chosen vertex in the graph is isolated tends to $1-q<1$.

\textit{Proof of Theorem \ref{betanot0}.}  Theorem \ref{betanot0} follows immediately from Propositions \ref{posrec}, \ref{tail1}, \ref{tail2} and \ref{as}.

\bibliographystyle{plain}
\bibliography{js}

\end{document}